\documentclass{amsart}
\usepackage{amssymb}
\usepackage[hyphens]{url}
\usepackage{amsmath}
\usepackage{amsthm}
\usepackage{thmtools}
\usepackage{mathrsfs}
\usepackage[utf8]{inputenc}
\usepackage[english]{babel}
\usepackage{cite}
\usepackage{hyperref}
\usepackage{cleveref}
\usepackage[section]{placeins}
\RequirePackage{textgreek}
\usepackage{enumitem}
\usepackage{bm}
\usepackage{nicefrac}
\usepackage{mathabx}
\usepackage[all,cmtip]{xy}
\usepackage{caption}
\usepackage{float}
\usepackage{mathtools}
\usepackage{xcolor}
\usepackage{tikz}
\title{Classification of Fuchsian groups with torsion}

\hypersetup{colorlinks=true, linkcolor=blue, citecolor=cyan, urlcolor=red}
\hypersetup{linktoc=all}
\numberwithin{equation}{subsection}
\newtheorem{thm}{Theorem}[section]
\newtheorem{lem}[thm]{Lemma}
\newtheorem{prop}[thm]{Proposition}
\newtheorem{cor}[thm]{Corollary}

\newtheorem{defn}[thm]{Definition}

\newtheorem{exm}[thm]{Example}

\newtheorem{que}[thm]{Question}
\newtheorem{rem}[thm]{Remark}

\newtheorem{fct}[thm]{Fact}

\DeclareMathOperator{\fg}{fg}

\DeclareMathOperator{\R}{\mathbb{R}}

\DeclareMathOperator{\Comm}{Comm}

\DeclareMathOperator{\PSL}{PSL}

\DeclareMathOperator{\supp}{supp}
\DeclareMathOperator{\Inc}{Inc}


\author{George Peterzil}
\address{Einstein Institute of Mathematics, The Hebrew University, Israel}
\email{george.peterzil@mail.huji.ac.il}

\thanks{The research was supported by the ISF (grant No. 957/25).}

\begin{document}

\begin{abstract}
     In their recent paper, Bergfalk and Smythe prove that the isometry equivalence relation on hyperbolic surfaces with finitely-generated fundamental group is concretely classifiable, and ask whether the same result holds true for 2-dimensional hyperbolic orbifolds, or equivalently, whether the action of $\PSL_2(\R)$ on its space of finitely-generated discrete subgroups is concretely classifiable. In this note we answer this question in the affirmative. We then use the result to prove that a nonsingular ergodic $\PSL_2(\R)$-space with nonelementary finitely-generated stabilizers is homogeneous, in similarity with a result of Stuck-Zimmer for lattices in semisimple lie groups. The main ingredients of our proof are Selberg's lemma and a result of Greenberg on commensurators.
\end{abstract}

\maketitle

\section{Introduction}

Given a linear group $H\leq \text{GL}_n(\mathbb{C})$ we write $\mathcal{S}(H)$ for the collection of closed subgroups of $H$. We equip $\mathcal{S}(H)$ with the Chabauty topology. This is a compact metrizable space, and in particular it is Polish. We write $\mathcal{S}_d(H)\subseteq \mathcal{S}(H)$ for the subspace consisting of discrete subgroups of $H$. The group $H$ acts on $\mathcal{S}(H)$ continuously by conjugation, such that $\mathcal{S}_d(H)$ is an invariant set. This space was introduced in \cite{Chabauty} and has since been widely used in various contexts. Descriptive set-theoretically, it is natural to ask for the complexity of the associated orbit equivalence relation, in general and when restricted to certain subspaces. 

In \cite{ManifoldClassification}, Bergfalk and Smythe provide a systematic framework for tackling manifold classification problems descriptive set-theoretically. In particular, they consider the problem of classification of hyperbolic surfaces up to isometry, or equivalently, the problem of classification of Fuchsian groups up to conjugacy. Using Teichm\"{u}ller theory, they prove the following theorem.

\begin{thm}[Theorem J in\cite{ManifoldClassification}]
\label{Fact: no torsion smooth}
    The action by conjugation of $\PSL_2(\R)$ on the space of torsion-free, finitely-generated Fuchsian groups is concretely classifiable \footnote{We follow their notational lead by using the term \emph{concretely classifiable} rather than the term \emph{smooth} in order to avoid confusion.}.
\end{thm}

 Somewhat surprisingly, the analogous statement for $\PSL_2(\mathbb{C})$ is proven false in the same paper (Theorem C). They also demonstrate that classification of discrete subgroups which are not necessarily finitely-generated is impossible in any noncompact semisimple lie group (Theorem I), following the unclassifiability of subgroups of the free group $F_2$ proved by Thomas and Velickovic \cite{ThomVeli} and of groups containing $F_2$ by Andretta, Camerlo and Hjorth \cite{AndrCameHjor}.

It is then asked whether the same is true without the assumption that the groups are torsion-free. We answer this question in the affirmative.

\begin{thm}
\label{Theorem: Main Theorem}
    The action by conjugation of $\PSL_2(\R)$ on the space of finitely-generated Fuchsian groups is concretely classifiable.
\end{thm}

\subsection*{Proof outline}

We use a result of Stuck--Zimmer on the classifiability of lattices to reduce the problem to the nonlattice case, where for elementary groups we construct a Borel transversal directly. For the nonelementary, nonlattice case, using Selberg's lemma \cite{Selberg}, we associate to each group $\Gamma$ the intersection $\Gamma_m$ of all of its subgroups of index $n_\Gamma$, where $n_\Gamma$ is the minimal index of a torsion-free subgroup of $\Gamma$, and using it we assign to a finitely-generated nonelementary Fuchsian group $\Gamma$ the group $\Comm_G(\Gamma)_m$ (where $\Comm_G(\Gamma)$ is the commensurator of $\Gamma$ in $\PSL_2(\R)$), which we prove induces a countable-to-one map between the conjugacy classes with torsion and those without torsion using a theorem of Greenberg on commensurators. The result of Bergfalk-Smythe tells us that the latter is concretely classifiable, and the theorem then follows from a descriptive set-theoretic argument.

The methods of Bergfalk and Smythe point at a natural reason why classifying torsion-free subgroups could be simpler; the problem of classifying \emph{manifolds} is apriori more accessible than the problem of classifying \emph{orbifolds}. Thus, we prove the reduction to the torsion-free case for certain classes of discrete subgroups of an arbitrary linear group (see \Cref{Definition: SCC}), hoping that it would prove useful.

\subsection*{A Stuck-Zimmer type corollary}

In their seminal paper \cite{StuckZimmer}, Stuck and Zimmer prove that the action of a lie group on its space of lattices is concretely classifiable. Using it, they furthermore prove:

\begin{thm}[Lemma 3.5 in\cite{StuckZimmer}]
\label{Fact: Stuck-Zimmer}
    Let $H$ be a semisimple lie group with no compact factors, and let $(X,\mu)$ be a nonsingular ergodic $H$-space such that for $\mu$-almost every $x\in X$ we have that the stabilizer $H_x$ is a lattice. Then there is a measure-class preserving, $H$-equivariant isomorphism $(X,\mu)\cong (H/\Gamma,\nu_\Gamma)$ for some lattice $\Gamma\leq H$.
\end{thm}

The paper \cite{StuckZimmer} has initiated the systematic study of probability measures on the space of subgroups, in particular for those invariant under conjugation, labeled Invariant Random Subgroups, or IRS for short (see \cite{lecture} for a short survey). This notion has proved incredibly fruitful (see, for example, \cite{AberGlasVira},\cite{BadeDucheLecu},\cite{7sam}).
In the particular case of $\PSL_2(\R)$, Biringer and Raimbault \cite{RaimBiri} provide a complete classification of the almost-sure topological type of  the quotient of an invariant random subgroup. For more general rank one groups, several interesting theorems and constructions are given in \cite{7sam2}. 

The bulk of the theory described above has been developed for conjugation-invariant probability measures on the space of subgroups, rather than for \emph{quasi}-invariant probability measures (with some exceptions, see \cite{FracGela} for a notable example). It is worth noting that very little is known at the full generality of quasi-invariant random subgroups beyond \Cref{Fact: Stuck-Zimmer} (although this terminology was introduced several years after their paper).

 Following the original proof of Stuck and Zimmer and using \Cref{Theorem: Main Theorem} we prove the theorem below.

\begin{thm}
\label{Theorem: Stuck-Zimmer}
    Let $(Y,\nu)$ be a nonsingular ergodic $\PSL_2(\R)$-space such that almost every stabilizer is finitely-generated and nonelementary. Then $(Y,\nu)$ is isomorphic to a homogeneous space $\PSL_2(\R)/\Gamma$, equipped with its canonical invariant measure class.
\end{thm}

\subsection*{Structure of the paper}

In section $2$ we review some preliminaries and basic results from descriptive set theory and from the theory of commensurators. In section $3$ we introduce the notion of a small commensurator class, in section $4$ we prove \Cref{Theorem: Main Theorem}, and in section $5$ we prove \Cref{Theorem: Stuck-Zimmer}.

\section*{Acknowledgments}

I thank Johanna Steinmeyer, who has tragically passed away as this project was nearing its completion, for her encouragement and support.

I also thank Nachi Avraham-Re'em, Jeffrey Bergfalk, Or Landesberg and Iian Smythe for enriching conversations and correspondences regarding this body of work.

\section{Preliminaries and basic facts}

\subsection*{Notation}

Given an equivalence relation $E$ on a set $X$ and some $X_0\subseteq X$, we write $E\restriction X_0$ for $E\cap X_0\times X_0$. We write $H$ for a general real linear algebraic group and $G$ for $\PSL_2(\R)$. Write $\mathcal{S}_{\fg}(H)\subseteq\mathcal{S}_d(H)$ for the subspace consisting of finitely-generated discrete subgroups of $H$, and $\tilde{\mathcal{S}}_{\fg}(H)\subseteq\mathcal{S}_{\fg}(H)$ for those which are torsion-free. For a subgroup $H_0\leq H$ and $g\in H$ we write $H^g_0=gH_0 g^{-1}$. We write $\mathcal{F}(H)\supseteq \mathcal{S}(H)$ for the space of closed \emph{subsets} of $H$, equipped with the Fell topology (see \cite{Nadler} for more on this topology).

\subsection*{Fundamentals from invariant descriptive set theory}

Let $(X,\beta)$ be a standard Borel space. Recall that an equivalence relation $E$ on $X$ is called \emph{Borel} if $E\in\beta\otimes \beta$ when considered as a set of pairs in $X$.

\begin{defn}
    Given standard Borel spaces $X,Y$ and Borel equivalence relations $F,E$ on $X,Y$ respectively, a \emph{Borel reduction from $E$ to $F$} is a Borel function $f:X\rightarrow Y$ such that for all $x_1,x_2\in X$ we have $(x_1,x_2)\in F$ if and only if $(f(x_1),f(x_2))\in E$. If there exists such a reduction we write $F\preceq E$ and say that \emph{$F$ is reducible to $E$}.
\end{defn}

\begin{defn}
    A Borel equivalence relation $E$ on a standard Borel space $X$ is called \emph{concretely classifiable} if there is a standard Borel space $Y$ such that $E\preceq \mathord{=}_Y$, or equivalently, if the quotient $\sigma$-algebra on $X/E$ induces a standard Borel structure on it.
\end{defn}

We have the following theorems of Kechris.

\begin{thm}[Theorems 7.5.1, 5.4.10 and 5.4.11 in \cite{Gao}]
\label{Fact: smooth OER has transversal}
    Let $H$ be a locally-compact, second-countable group, and suppose $a:H\times X\rightarrow X$ is a Borel action of $H$ on a standard Borel space $X$ with an associated orbit equivalence relation $E$.
    
    \begin{itemize}
        \item There is a countable Borel cross-section for $E$; namely, a Borel subset $X_0\subseteq X$ which intersects each $E$-class on a nonempty countable set.
        \item If, in addition, $E$ is concretely classifiable, then there is a Borel transversal for $E$: a Borel set $T\subseteq X$ which intersects every $E$-class at exactly one point, and it can be taken to be a subset of $X_0$.
    \end{itemize}
\end{thm}

The following is classical.

\begin{thm}[Luzin-Novikov uniformization, Theorem 18.10 in \cite{Kechris}]
\label{Fact: Luzin-Novikov}
    Let $X,Y$ be standard Borel spaces, $R\subseteq X\times Y$ Borel such that for every $x\in X$, the right fiber $R_x$ is countable and nonempty. Then there is a sequence of Borel functions $F_k:X\rightarrow Y$ such that $R=\bigcup_{k<\infty}Graph(F_k)$. 
\end{thm}

Uniformization has several useful corollaries.

\begin{thm}[Exercise 18.14 in \cite{Kechris}]
\label{Fact: countabe-to-one image}
    Let $f:X\rightarrow Y$ be a countable-to-one Borel function between standard Borel spaces. Then $f(A)\subseteq Y$ is Borel for any Borel set $A\subseteq X$.
\end{thm}

We also use the following descriptive set-theoretic lemma, which could be interesting in its own right.

\begin{lem}
\label{Proposition: countable-to-smooth is enough}
    Let $X,Y$ be standard Borel spaces, $F$ a Borel orbit equivalence relation on $X$ which is induced by a Borel locally compact group action, $E$ a concretely classifiable Borel orbit equivalence relation on $Y$, $f:X\rightarrow Y$ a countable-to-one Borel map such that if $(x,y)\in F$ then $f(x)=f(y)$. Assume in addition that the set-theoretic map $\hat{f}:X/F\rightarrow Y/E$ is countable-to-one. Then $F$ is concretely classifiable.
\end{lem}

\begin{proof}
    Apply \Cref{Fact: smooth OER has transversal} to get a countable Borel cross-section $X_0\subseteq X$ for $F$. By \Cref{Fact: countabe-to-one image} the set $Y_0=f(X_0)$ is Borel. We also observe that since $\hat{f}$ is countable-to-one, the equivalence relation $E_0:=E\restriction Y_0$ has countable classes. Since $E_0$ inherits classifiability from $E$, it admits a Borel transversal $T\subseteq Y_0$. We now apply \Cref{Fact: Luzin-Novikov} to the equivalence relation $E_0$ and get a sequence of Borel functions $g_n:Y_0\rightarrow Y_0$ which span $E_0$. Write $F_0=f_*F\restriction Y_0$ for the pushforward equivalence relation restricted to $Y_0$. Define recursively $T_0=T$ and:

    \[T_{k+1}=T_k\cup \left\{g_k(t)\mid t\in T,g_k(t)\notin [T_k]_{F_0}\right\}\]

    Where $[T_k]_{F_0}=\left\{y\in Y_0\mid \exists x\in T_k\right (x,y)\in F_0\}$. Note that $[T_k]_{F_0}$ is a Borel set, as it is just equal to $\bigcup_{n<\infty} g_n(T_k)$, hence $T_{k+1}$ is Borel as well. We thus get that the set $T'=\bigcup T_k$ is a Borel transversal for $F_0$, and so the map taking $x\in X$ to the unique point $s(x)\in T'$ which is $f_*F$-equivalent to $f(x)$ is Borel, hence it witnesses that $F$ is concretely classifiable.
\end{proof}

\subsection*{Descriptive set theory on the space of subgroups}

Recall the following two standard facts.

\begin{fct}
With $H$ as above we have the following.
    \begin{itemize}
        \item The intersection map $\cap:\mathcal{S}_{\fg}(H)\times\mathcal{S}_{\fg}(H)\rightarrow\mathcal{S}_d(H)$ is Borel.
        \item The subgroup relation $\Inc=\left\{(\Gamma_1,\Gamma_2)\in \mathcal{S}_{\fg}(H)\times \mathcal{S}_{\fg}(H)\mid \Gamma_1\leq \Gamma_2\right\}$ is Borel.
    \end{itemize}
\end{fct}

We include an additional fundamental property of Chabauty space.

\begin{lem}
    For any $n\in\mathbb{N}$, write $\Inc_n(H)=\left\{(\Gamma,\Gamma')\in\Inc\mid [\Gamma':\Gamma]=n\right\}$. Then $\Inc_n(H)\subseteq\mathcal{S}(H)\times\mathcal{S}(H)$ is locally closed, i.e. there are closed sets $C_1,C_2\subseteq \mathcal{S}_{\fg}(G)\times \mathcal{S}_{\fg}(G)$ such that $\Inc_n(H)=C_1\setminus C_2$. In particular, $\Inc_n(H)$ is a Borel set.
\end{lem}

\begin{proof}
    It suffices to prove that $\bigcup_{k<n}\Inc_n(H)$ is closed. Recall that $\mathcal{S}(H)\subseteq\mathcal{F}(H)$ is a topological subspace inclusion.  This space is equipped with the continuous action given by \emph{multiplication on the left}. Suppose $\left((\Gamma_n,\Gamma'_n)\right)_{n\in\mathbb{N}}$ is a sequence in $\Inc_m(H)$ which converges to a pair $(\Gamma,\Gamma')$, and suppose $g_1,...,g_{m+1}\in \Gamma'$ are $\Gamma$-inequivalent. Since $\Gamma_n\rightarrow \Gamma$ we also have $g_ig_j^{-1}\Gamma_n\rightarrow g_ig_j^{-1}\Gamma$ for all $i\neq j$. Pick an compact identity neighborhood $K$ in $H$ such that $\Gamma'\cap K=\{e\}$, and in particular $g_ig_j^{-1}\Gamma\cap K=\emptyset$ for $i<j$. Thus, for $N\gg 0$ we have $g_ig_j^{-1}\Gamma_N\cap U=\emptyset$ for $i\neq j$, and since $e\in K$ we get $g_jg_i^{-1}\notin \Gamma_N$, meaning that for $N\gg0$ we have that $g_1,...,g_{m+1}$ are pairwise $\Gamma_N$-inequivalent, a contradiction.
\end{proof}

We write $\text{Inc}_{<\infty}(H)=\bigcup_{n<\infty}\text{Inc}_n(H)$.

\subsection{Commensurability and commensurators}

Recall the following definition.

\begin{defn}
    Two subgroups $H_1,H_2\leq H$ are called \emph{commensurable} if $H_1\cap H_2$ is of finite index in $H_1$ and in $H_2$.
\end{defn}

This gives rise to the following refinement of the normalizer of a subgroup.

\begin{defn}
     Given a subgroup $H_0\leq H$, we define the \emph{commensurator of $H_0$ in $H$} as follows:

    \[\Comm_H(H_0)=\left\{g\in H\mid H_0,H^g_0\text{ are commensurable}\right\}\]
\end{defn}

Note that this is a subgroup of $H$ which contains the normalizer $N_H(H_0)$, and in particular it contains $H_0$. 

We now prove an elementary fact about commensurators.

\begin{prop}
\label{Proposition: commensurability implies commensurator}
    Suppose $\Gamma_1,\Gamma_2\leq H$ are commensurable, then we have the equality $\Comm_{H}(\Gamma_1)=\Comm_H(\Gamma_2)$. In particular, if $[\Comm_H(\Gamma_1):\Gamma_1]<\infty$ then $\Comm_H(\Comm_H(\Gamma_1))=\Comm_H(\Gamma_1)$.
\end{prop}

\begin{proof}
    Since having equal commensurators is an equivalence relation on the set of subgroups, it is sufficient to prove that if $\Gamma_1\leq \Gamma_2$ is of finite index then $\Comm_H(\Gamma_1)=\Comm_H(\Gamma_2)$. Thus, suppose $\Gamma_1\leq \Gamma_2$ have $[\Gamma_2:\Gamma_1]<\infty$. If $g\in \Comm_H(\Gamma_1)$ then:

    \[[\Gamma_2:\Gamma_2\cap \Gamma_2^g]\leq [\Gamma_2:\Gamma_1\cap \Gamma_1^g]= [\Gamma_2:\Gamma_1]\cdot [\Gamma_1:\Gamma_1\cap \Gamma_1^g]\]\\[-3ex]
    
    Suppose now that $g\in \Comm_G(\Gamma_2)$. First note that we have:

    \[[\Gamma_2\cap \Gamma^g_2:\Gamma_1\cap \Gamma_1^g]=[\Gamma_2\cap \Gamma_2^g:\Gamma_2\cap \Gamma_1^g]\cdot [\Gamma_2\cap \Gamma_1^g:\Gamma_1\cap \Gamma_1^g]\leq [\Gamma_2^g:\Gamma_1^g]\cdot [\Gamma_2:\Gamma_1]<\infty\]\\[-3ex]

    This gives us:

    \[[\Gamma_1:\Gamma_1\cap \Gamma_1^g]\leq [\Gamma_2:\Gamma_1\cap \Gamma_1^g]=[\Gamma_2:\Gamma_2\cap \Gamma_2^g]\cdot [\Gamma_2\cap \Gamma_2^g:\Gamma_1\cap \Gamma_1^g]<\infty\]
\end{proof}

The study of commensurators of discrete subgroups of lie groups is an active area of research (see \cite{Marg1},\cite{CreuShal},\cite{Shal}). We recall the following classical result of Greenberg, which is a crucial ingredient in our proof:

\begin{thm}[Theorem 2(3) and 2(4) in\cite{FuchsianCommensurators}]
\label{Fact: Discrete+finite index commensurators}
    Let $\Gamma\leq G$ be a finitely-generated nonelementary Fuchsian group which is not a lattice. Then $\Comm_G(\Gamma)$ satisfies $[\Comm_G(\Gamma):\Gamma]<\infty$ (and in particular, is discrete).
\end{thm}

In particular, $\Comm_G(\Gamma)$ is finitely-generated in this case.

\section{Small Commensurator Classes}

We now isolate a key property for a family of finitely-generated discrete subgroups of $G$.

\begin{defn}
\label{Definition: SCC}
    Call a nonempty, $H$-invariant Borel set $X\subseteq \mathcal{S}_{\fg}(H)$ a \emph{small commensurator class} (SCC for short) if it satisfies the following two conditions:

    \begin{itemize}
        \item If $\Gamma\in X$ and $\Gamma$ is commensurable to some $\Gamma_0$ then $\Gamma_0\in X$.
        \item If $\Gamma\in X$ then $[\Comm_H(\Gamma):\Gamma]<\infty$.
    \end{itemize}
\end{defn}

\begin{rem}
    Note that any $\Gamma$ in an SCC has that $\Comm_H(\Gamma)$ is necessarily finitely-generated and discrete.
\end{rem}

\begin{exm}
\label{Example: Fuchsian}
    Write $\mathcal{S}_{\fg}(\PSL_2(\R))'\subseteq \mathcal{S}_{\fg}(\PSL_2(\R))$ for the nonelementary finitely-generated Fuchsian groups which are not lattices. By \Cref{Fact: Discrete+finite index commensurators} we have that $\mathcal{S}_{\fg}(G)'$ is an SCC.
\end{exm}

For an SCC, we have the following inverse to \Cref{Proposition: commensurability implies commensurator}:

\begin{prop}
\label{Proposition: commensurator implies commensurability}
    Let $X\subseteq\mathcal{S}_{\fg}(H)$ be an SCC, and fix $\Gamma_1,\Gamma_2\leq H$ such that $\Comm_H(\Gamma_1)=\Comm_H(\Gamma_2)$, then $\Gamma_1,\Gamma_2$ are commensurable.
\end{prop}

\begin{proof}
    Since $X$ is an SCC, any member of it is commensurable to its commensurator. The proposition follows from the fact that commensurability is an equivalence relation.
\end{proof}

The following lemma characterizes commensurators in SCCs.

\begin{lem}
\label{Lemma: maximality of commensurators}
    Let $X$ be an SCC and fix $\Gamma,\Gamma'\in X$. Then $\Gamma'=\Comm_H(\Gamma)$ if and only if $\Gamma\leq \Gamma'$ with finite index, and for every $\Gamma_0\geq \Gamma$ with $[\Gamma_0:\Gamma]<\infty$ we have $\Gamma_0\leq \Gamma'$.
\end{lem}

\begin{proof}
    Let $X,\Gamma,\Gamma'$ be as in the lemma. The first direction is immediate (and does not require $X$ to be an SCC). For the second, note first that since $X$ is an SCC we have that $\Gamma\leq \Comm_H(\Gamma)$ is of finite-index, hence by the stated property of $\Gamma'$ we have $\Comm_H(\Gamma)
    \leq \Gamma'$. For the other inclusion, note that any element of a group containing $\Gamma$ with finite index commensurates it.
\end{proof}

Given an SCC $X$, write $X_{\max}$ for the collection of commensurators in $X$ and $\tilde{X}$ for the collection of torsion-free subgroups in $X$ (note that the latter is a Borel set by Lemma 5.1 of \cite{ManifoldClassification}).

\begin{prop}
\label{Proposition: commensurator is measurable}
    Given an SCC $X$, the map $\Gamma\mapsto \Comm_G(\Gamma)$ is Borel-measurable as a function between $X$ and itself, and the set $X_{\max}\subseteq X$ is a Borel set.
\end{prop}

\begin{proof}
    The set $\Inc_{<\infty}(H)$ has countable left fibers, so we apply \Cref{Fact: Luzin-Novikov} to get a sequence Borel functions $F_n:X\rightarrow X$ with $\text{Inc}_n(G)=\bigcup_{n<\infty}Graph(F_n)$. We now construct a new sequence of functions $\tilde{F}_n=:X\rightarrow X$ recursively; we set $\tilde{F}_0=F_0$ and $\tilde{F}_{n+1}(\Gamma)$ is equal to $F_{n+1}(\Gamma)$ if $\tilde{F}_n(\Gamma)\subseteq F_{n+1}(\Gamma)$, and otherwise equal to $\tilde{F}_n(\Gamma)$. We note that for any $\Gamma$ the sequence $[\Comm_G(\Gamma):\tilde{F}_n(\Gamma)]$ is non-decreasing, and moreover since $X$ is an SCC there exists some $n=n_\Gamma$ such that $F_n(\Gamma)=\Comm_G(\Gamma)$, by \Cref{Lemma: maximality of commensurators} we also have $\tilde{F}_{n_\Gamma}(\Gamma)=\Comm_G(\Gamma)$. We get that the Borel function $\Gamma\mapsto \tilde{F}_{n_\Gamma}(\Gamma)$ equals the desired function, finishing the proof of the first part. For the second part, by \Cref{Fact: countabe-to-one image} and by the second part of \Cref{Proposition: commensurability implies commensurator} we have that $\Comm_H(X)=X_{\max}$ is Borel.
\end{proof}

\begin{rem}
    One can interpret \Cref{Proposition: commensurability implies commensurator}, \Cref{Proposition: commensurator implies commensurability} and \Cref{Proposition: commensurator is measurable} together as follows: for an SCC $X$, the map $\eta:X\rightarrow X_{\max}$ defined by $\Gamma\mapsto \Comm_H(\Gamma)$ is a Borel reduction from the equivalence relation of commensurability to that of equality. Moreover, since conjugation by $h\in H$ maps the commensurator of $\Gamma\leq H$ to the commensurator of $\Gamma^h$, this map is also a reduction from the equivalence relation of commensurability of conjugates to that of conjugacy.
\end{rem}

We now deduce the following.

\begin{cor}
\label{Corollary: first reduction}
    For an SCC $X$, the map $\Gamma\mapsto \Comm_G(\Gamma)$ maps conjugate subgroups to conjugate subgroups, and the induced map from the set of conjugacy classes of elements of $X$ to those in $X_{\max}$ is countable-to-one. In particular, if the conjugacy equivalence relation on $X_{\max}$ is concretely classifiable then so is conjugacy on $X$.
\end{cor}

\begin{proof}
    The first and second parts are immediate. The third part follows from \Cref{Proposition: commensurator is measurable} and from \Cref{Proposition: countable-to-smooth is enough}.
\end{proof}

\section{Eliminating torsion in small commensurability classes}

Recall the following Lemma, due to Selberg.

\begin{fct}[Selberg's lemma,\cite{Selberg}]
\label{Fact: Selberg's lemma}
    Let $\Gamma\leq GL_n(\mathbb{C})$ be finitely-generated, then $\Gamma$ contains a torsion-free subgroup of finite index.
\end{fct}

Let $X\subseteq\mathcal{S}(H)$ be a small commensurability class. We wish to use \Cref{Fact: Selberg's lemma} to construct an $H$-equivariant Borel function $\eta:X_{\max}\rightarrow \tilde{X}$ such that for every $\Gamma\in X$ we have that $\eta(\Gamma)\leq \Gamma$ is of finite index. We do so as follows.

\begin{defn}
    For $\Gamma\in X$, we write $n_\Gamma$ for the minimal $n\in\mathbb{N}$ such that $\Gamma$ contains a torsion-free subgroup of index $n$ (which exists by Selberg's lemma). We then define:
    
    \[\Gamma_m=\bigcap \{\Gamma_0\leq \Gamma\mid [\Gamma:\Gamma_0]=n_\Gamma\}\]

\end{defn}

We have the following.

\begin{prop}
\label{Lemma: Selberg equivariant Borel}
    The function $\Gamma\mapsto\Gamma_m$ satisfies the following properties:
    \begin{itemize}
        \item The subgroup $\Gamma_m$ torsion-free and of finite index in $\Gamma$.
        \item The function is Borel.
        \item The function is $H$-equivariant.
        \item The function restricts to a reduction from the equivalence relation of conjugacy on $X_{\max}$ to that of conjugacy on $\tilde{X}$.
    \end{itemize}
\end{prop}

\begin{proof}
    \begin{itemize}
        \item The first part follows from the definition of $n_\Gamma$, and the second from the fact that a finitely-generated group contains only finitely many subgroups of a given index, hence their intersection is of finite index in $\Gamma$.
        \item We first prove that the assignment $\Gamma\mapsto n_\Gamma$ is Borel, which is equivalent to proving that the collection of subgroups whose minimal index of a torsion-free subgroup is at most $n$ is Borel. Write $J_n$ for this collection of subgroups. We then have the following:
        
        \[J_n=\pi_1\left(X\times X'\cap \text{Inc}_n(H)\right)\]

        Where $\pi_1$ is the projection onto the first coordinate. This projection is finite-to-one, and by \Cref{Fact: countabe-to-one image} we get that $J_n$ is Borel.
        
        Observe $R=\text{Inc}_n(H)\cap J_n\times \tilde{X}$. The relation $R$ has finite right-fibers, hence by \Cref{Fact: Luzin-Novikov}  we can write $R^=\bigcup _{k<\infty} Graph(F_k)$ for some Borel functions $F_k:J_n\rightarrow X'$. Consider now the Borel maps $\tilde{F}_k:J_n\rightarrow [\tilde{X}]^{<\infty}$ which are defined recursively by $\tilde{F}_0=F_0$ and $\tilde{F}_{k+1}(\Gamma)=\tilde{F}_k(\Gamma)\smallfrown(F_{k+1}(\Gamma))$ if $F_{k+1}(\Gamma)\neq F_i(\Gamma)$ for $i<k+1$, and otherwise $\tilde{F}_{k+1}(\Gamma)=\tilde{F}_k(\Gamma)$ (that is, we list the subgroups of index $n$ in $\Gamma$ according to the ordering given by the maps $\tilde{F}_i$, without repetition). We finally set $F:X'\rightarrow [X']^{<\infty}$ by $F(\Gamma)=F_k(\Gamma)$, where $k$ is minimal such that $F_k(\Gamma)=F_{N}(\Gamma)$ for all $N>k$. Note that since $[X']^{<\infty}=\bigcup _{k<\infty} (X')^k$ and since the map $\cap:[X']^{<\infty}\rightarrow X'$ is Borel by a simple induction argument, the map $S=\cap\circ F$ is Borel.
        \item This follows from the fact that for any $g\in H$, $[\Gamma:\Gamma']=k$ if and only if $[\Gamma^g:\left(\Gamma'\right)^g]=k$.
        \item Suppose $\Gamma_1,\Gamma_2\in X'$ and write $\Lambda_i=\Comm_H(\Gamma_i)$. If $\Lambda^g_2=\Lambda_2$ for some $g\in H$ then by the above $(\Lambda_2)^g_m=(\Lambda_1)_m$. For the other direction, suppose $(\Lambda_2)^g_m=(\Lambda_1)_m$ for some $g\in H$. Since $\Lambda_1,\Lambda_2$ are finitely-generated we have that $\Lambda_i$ is commensurable to $(\Lambda_i)_m$, and by the same theorem and transitivity we have that $\Gamma_i$ is commensurable to $(\Lambda_i)_m$, for $1\leq i \leq 2$. Since commensurability is preserved under conjugation, we get that $\Gamma_1$ is commensurable to $\Gamma^g_2$, hence $\Lambda_1=\Comm_H(\Gamma_1)=\Comm_H(\Gamma^g_2)$ by \Cref{Proposition: commensurability implies commensurator}. But $\Comm_H(\Gamma^g_2)=\Comm_H(\Gamma_2)^g=\Lambda^g_2$, finishing the proof. 
    \end{itemize}
\end{proof}

We get the following immediate corollary from \Cref{Lemma: Selberg equivariant Borel} and \Cref{Corollary: first reduction}.

\begin{cor}
\label{Corollary: torsion eliminated in SCC}
    For an SCC $X\subseteq \mathcal{S}_{\fg}(H)$ such that conjugacy on $\tilde{X}$ is concretely classifiable we have that conjugacy on $X$ is concretely classifiable. 
\end{cor}

\subsection*{The case of Fuchsian groups}

We now focus on $G=\PSL_2(\R)$. We denote by $X_{\text{e}},X_{\text{l}}\subseteq\mathcal{S}_{fg}(G)$ for the collections of elementary Fuchsian groups and lattices, respectively (recall that a Fuchsian group is called elementary if it is virtually cyclic, or equivalently, if its limit set is finite, see \cite{DalBo}). As before, we write $\mathcal{S}_{\fg}(G)'$ for the collection of nonelementary, finitely-generated, nonlattice Fuchsian groups. Note that since the assignment of a limit set to a Fuchsian group is Borel (see Theorem 8.20 in \cite{ManifoldClassification}), and since a finitely-generated Fuchsian group is a lattice if and only if it has full limit set, the sets $X_{\text{e}},X_{\text{l}}$ are Borel. 

\begin{rem}
\label{Remark: elemntary classif}
    Constructing a Borel transversal for conjugacy on $X_{\text{e}}$ is can be done directly. We follow Theorem 2.7.5 in \cite{FuchsianLectures}, and note that among the representatives for conjugacy presented there, the only pairwise conjugate ones are the infinite cyclic groups which are generated by unipotent elements.
\end{rem}

We can thus prove the main theorem.

\begin{proof}[Proof of \Cref{Theorem: Main Theorem}]
    Consider the following conjugacy-invariant Borel decomposition $\mathcal{S}_{\fg}(\PSL_2(\R))=\mathcal{S}_{\fg}(\PSL_2(\R))'\sqcup X_{\text{l}}\sqcup X_{\text{e}}$. It is sufficient to prove that conjugacy is concretely classifiable on each of the parts. The first component is an SCC by \Cref{Example: Fuchsian} and so conjugacy is concretely classifiable on it by \Cref{Fact: no torsion smooth} and \Cref{Corollary: torsion eliminated in SCC}. By Corollary 3.2 in \cite{StuckZimmer} we have concrete classification on the second component, and the third is concretely classifiable by \Cref{Remark: elemntary classif}. 
\end{proof}

\begin{rem}
    We note that the above proof ``by parts" is not artificial; the argument we have provided for nonelementary groups does not work for elementary groups, as commensurability of conjugates on the space of elementary groups admits a reduction from the orbit equivalence relation associated to the natural action $\mathbb{Q}_{>0}^{\times}\curvearrowright \mathbb{R}_{>0}^\times$, whose orbit equivalence relation is not concretely classifiable by Vitali's argument. This reduction is given as follows:
    
    \[r\mapsto \left\langle\begin{pmatrix}
        e^{\frac{r}{2}} & 0\\
        0 & e^{-\frac{r}{2}}
    \end{pmatrix}\right\rangle\]
\end{rem}

We conclude with the following natural question.

\begin{que}
    Is the action of $\PSL_2(\R)\times \PSL_2(\R)$ on the space of its finitely-generated discrete subgroups concretely classifiable?
\end{que}

\section{Quasi-invariant measures}

We write $X'$ for the collection of nonelementary, nonlattice Fuchsian groups, and write $G=\PSL_2(\R)$.

We first prove a direct corollary of concrete classification. This argument is attributed to Glimm-Effros.

\begin{cor}
\label{Corollary: QIRS}
    For a probability measure $\nu$ on $\mathcal{S}_{\fg}(G)$ which is quasi-invariant and ergodic for the action of $G$, there exists some $\Gamma\in\mathcal{S}_{\fg}(G)$ with $\nu(\Gamma^G)=1$.
\end{cor}

\begin{proof}
    Let $\pi:\mathcal{S}_{\fg}(G)\rightarrow \mathcal{S}_{\fg}(G)/G$ denote the quotient map. Since the latter is standard Borel and $\pi$ is $G$-invariant, it must be constant $\nu$-almost surely, hence $\nu$ is supported on the $G$-orbit of its almost-sure value $\Gamma$.
\end{proof}

The following proof is carried out as in Lemma 3.5 of \cite{StuckZimmer}, with some additional details provided.

\begin{proof}[Proof of \Cref{Theorem: Stuck-Zimmer}]
    Without loss of generality we may assume that every stabilizer is in $X'$. Let $\nu^*$ be the pushforward measure on $X'$. By \Cref{Corollary: QIRS}, it is supported on a single orbit $\Gamma^G$.

    Let $\left\{\nu_{\Gamma^g}:g\in G\right\}$ denote a disintegration of $\nu$ along $S$. By uniqueness, equivariance of the stabilizer map and quasi-invariance of $\nu$, for every $g\in G$ and $\nu^*$-almost every $\Gamma^h\in \Gamma^G$ we have $g^*\nu_{\Gamma^h}\sim\nu_{\Gamma^{gh}}$. Observe that for every $g\in G$, the group $N_G(\Gamma^g)$ preserves the fiber $S^{-1}(\Gamma^g)$, hence the probability space $(S^{-1}(\Gamma^g),\nu_{\Gamma^g})$ is a nonsingular $N_G(\Gamma)/\Gamma$-system. We claim that it is ergodic for every $g\in G$. Indeed, if $f_{\Gamma^g}\in L^1(\nu_{\Gamma^g})$ is $N_G(\Gamma^g)/\Gamma^g$-invariant, then for every $h\in G$ we can define $f_{\Gamma^{hg}}\in L^1(\nu_{\Gamma^{hg}})$ bu $f_{\Gamma^{hg}}=f_{\Gamma^g}\circ h^{-1}$, and note that by the $N_G(\Gamma)$-invariance of $f_{\Gamma^g}$ it is well-defined. We now define $f\in L^1(\mu)$ by setting $f$ to restrict to each $S^{-1}(\Gamma^g)$ as $f_{\Gamma^g}$. This function is $G$-invariant, hence constant, thus so is $f_{\Gamma^g}$ and the system is indeed ergodic.
    
    Note that $N_G(\Gamma)/\Gamma$ is a finite group (this follows from \Cref{Fact: Discrete+finite index commensurators} for example), hence ergodicity implies that $S^{-1}(\Gamma^g)$ is a transitive $N_G(\Gamma^g)/\Gamma^g$-set, up to $\nu_{\Gamma^g}$-measure $0$. Moreover, this action is free, since the underlying space is $S^{-1}(\Gamma^g)$. 
    
    Fixing $y\in\supp(\nu_\Gamma)$, the above argument demonstrates that the map $\eta:G/\Gamma\rightarrow Y$ given by $\eta(g\Gamma)\rightarrow g.y$ is well-defined, $G$-equivariant, injective and with $\nu$-conull image. By uniqueness of the $G$-invariant measure class $[m]$ we have $(\eta^{-1})^*\nu\sim m$.
\end{proof}

\bibliographystyle{acm}
\bibliography{Bibliography}

\end{document}